\documentclass[12pt]{amsart}

\usepackage{mathscinet}

\usepackage{graphicx}

\newtheorem{Theorem}{Theorem}
\newtheorem{thmalpha}{Theorem}
\newtheorem{Lemma}{Lemma}
\newtheorem{Proposition}{Proposition}
\newtheorem{Corollary}{Corollary}
\theoremstyle{definition}

\newtheorem{Problem}{Problem}
\theoremstyle{remark}
\newtheorem{Remark}{Remark}

\newtheorem*{acknowledgement}{Acknowledgement}

\begin{document}


\title[Stability of the Kaczmarz Reconstruction]{Stability of the Kaczmarz Reconstruction for Stationary Sequences}
\author{Caleb Camrud}
\author{Evan Camrud}
\author{Lee Przybylski}
\author{Eric S. Weber}
\address{Department of Mathematics, Iowa State University, 396 Carver Hall, Ames, IA 50011}
\email{ccamrud@iastate.edu}
\email{ecamrud@iastate.edu}
\email{leep@iastate.edu}
\email{esweber@iastate.edu}
\date{\today}
\begin{abstract}
The Kaczmarz algorithm is an iterative method to reconstruct an unknown vector $f$ from inner products $\langle f , \varphi_{n} \rangle $.  We consider the problem of how additive noise affects the reconstruction under the assumption that  $\{ \varphi_{n} \}$ form a stationary sequence.  Unlike other reconstruction methods, such as frame reconstructions, the Kaczmarz reconstruction is unstable in the presence of noise.  We show, however, that the reconstruction can be stabilized by relaxing the Kaczmarz algorithm; this relaxation corresponds to Abel summation when viewed as a reconstruction on the unit disc.  We show, moreover, that for certain noise profiles, such as those that lie in $H^{\infty}(\mathbb{D})$ or certain subspaces of $H^{2}(\mathbb{D})$, the relaxed version of the Kaczmarz algorithm can fully remove the corruption by noise in the inner products.  Using the spectral representation of stationary sequences, we show that our relaxed version of the Kaczmarz algorithm also stabilizes the reconstruction of Fourier series expansions in $L^2(\mu)$ when $\mu$ is singular.
\end{abstract}
\maketitle



	\section{Introduction}

The Kaczmarz algorithm \cite{K-37} is an iterative procedure for solving a system of linear equations.  We formulate the system of linear equations as inner products: given vectors $\{ \varphi_{0}, \dots, \varphi_{N-1}\} \subset \mathbb{C}^d$, find $\vec{x} \in \mathbb{C}^{d}$ such that $\langle \vec{x}, \varphi_{j} \rangle = b_{j}$.  The Kaczmarz algorithm proceeds as follows: given a solution guess $\vec{x}_n$ and an equation number $i$, we calculate
$r_i = b_i - \langle \vec{x}_n, \varphi_i \rangle$ (the residual for equation $i$), and define
\begin{equation}\label{eq:update}
  \vec{x}_{n+1} = \vec{x}_n + \frac{r_i}{\| \varphi_i \|^2} \varphi_i.
\end{equation}
This makes the residual of $\vec{x}_{n+1}$ in equation $i$ equal to 0. We then iterate repeatedly through all equations (i.e. $\lim_{n \to \infty} \vec{x}_{n}$ where $n+1 \equiv i \mod N$).  Kaczmarz proved that if $\{ \varphi_{0}, \dots, \varphi_{N-1}\}$ span $\mathbb{C}^d$ and the equations $\langle \vec{x}, \varphi_{j} \rangle = b_{j}$ are consistent, then $\vec{x}_{n} \to \vec{x}$.  If the system is inconsistent, then the relaxed version of the Kaczmarz algorithm can be utilized to obtain the least-squares solution (see Subsection \ref{ssec:RK-LSS} for details).

The Kaczmarz algorithm can also be used to reconstruct a vector (signal) in an infinite dimensional Hilbert space $H$ from linear measurements.  For a sequence of unit vectors $\{ \varphi_{n} \}_{n=0}^{\infty} \subset H$ ($H$ is of any dimension, finite or infinite), the sequence is said to be \emph{effective} if for every $\vec{x} \in H$, given the data $\{ \langle \vec{x}, \varphi_{n} \rangle \}_{n=0}^{\infty}$, the sequence of approximations $\vec{x}_{n}$ to $\vec{x}$ given by the Kaczmarz algorithm as defined in Equation (\ref{eq:update}) has the property that $\lim_{n \to \infty} \| \vec{x} - \vec{x}_{n} \| = 0$.  

In an infinite dimensional Hilbert space, the question of when a sequence of vectors is an effective sequence is subtle.  There are two main existing results concerning the effectivity of sequences.  The full characterization is in \cite{HalSzw05} (see also \cite{Szw06,CzaTan13a}) and is given in terms of when a certain infinite matrix defines a partial isometry on $\ell^2(\mathbb{N}_{0})$.  Another  characterization concerns stationary sequences, which is our focus here.  A sequence $\{ \varphi_{n}\}_{n=0}^{\infty}$ is stationary provided for every $n,m,k \in \mathbb{N}_{0}$, $\langle \varphi_{n + k} , \varphi_{m + k} \rangle = \langle \varphi_{n}, \varphi_{m} \rangle$.  Kwapien and Mycielski  \cite{KwMy01} prove the following:  the stationary sequence $\{ \varphi_{n} \}_{n=0}^{\infty} \subset H$ is effective if and only if it is linearly dense and the spectral measure of the sequence is either Lebesgue measure or purely singular.  

We consider the problem of stability, in the presence of additive noise, of the Kaczmarz reconstruction for stationary sequences which are effective:  
\begin{Problem}  \label{P:problem-statement}
Suppose $\{\varphi_{n} \}_{n=0}^{\infty} \subset H$ is a stationary, effective sequence.  Given the data $\{ \langle \vec{x}, \varphi_{n} \rangle + \epsilon_{n} \}_{n=0}^{\infty}$, does the Kaczmarz algorithm provide a reconstruction of $\vec{x}$?  Explicitly, the update becomes
\begin{equation} \label{eq:noise-update}
\vec{x}_{n+1} = \vec{x}_{n} + \left( \langle \vec{x}, \varphi_{n+1} \rangle + \epsilon_{n+1} - \langle \vec{x}_{n} , \varphi_{n+1} \rangle \right) \varphi_{n+1}.
\end{equation}
Does the sequence $\{ \vec{x}_{n} \}$ of approximations converge?  If so, what is the error in the reconstruction contributed by the noise $\{ \epsilon_{n} \}_{n=0}^{\infty}$: $\lim_{n \to \infty} \| \vec{x} - \vec{x}_{n} \|$?  
\end{Problem}

As we will demonstrate, Problem \ref{P:problem-statement} can be reformulated in terms of analytic functions on $\mathbb{D}$.  In particular, the characterization given by Kwapien and Mycielski uses the Herglotz Representation, which demonstrates that the Poisson transform induces a one-to-one correspondence between singular measures on $\mathbb{T}$ and inner functions on $\mathbb{D}$.  In \cite{HW17a}, the proof in \cite{KwMy01} is restated in terms of single variable operator theory: Beurling's theorem on the shift invariant subspaces of the Hardy space $H^2(\mathbb{D})$ \cite{beu48a}, the spectral theorem of rank-one perturbations of the shift \cite{Clark72}, and the Normalized Cauchy Transform \cite{Pol93}.  Subsequent work \cite{HJW18a} demonstrated that the Kaczmarz algorithm can be used to construct reproducing kernels within the Hardy space that have prescribed boundary behavior.
	
The boundary behavior arises from the spectral representation of stationary sequences (see Subsection \ref{ssec:HS-BF}). From this spectral representation, we are able to translate the problem onto $\mathbb{D}$.  Our main results will be to show that the reconstruction with noisy measurements in Equation \eqref{eq:noise-update} can be stabilized by employing Abel summation, e.g. Theorems \ref{Th:overview}, \ref{Th:epsilon-boundary} and \ref{Th:epsilon-bounded}, which corresponds to the boundary data of an analytic function on $\mathbb{D}$.  We will describe how Abel summation can be given in terms of an augmented Kaczmarz algorithm in Theorem \ref{Th:Aug-Kacz}.  We will then be able to use powerful results concerning the boundary behavior of the model subspaces of $H^2(\mathbb{D})$ (i.e., the subspaces which are backward shift invariant \cite{beu48a,Clark72,Aleks89a,Pol93}) and the machinery of sub-Hardy Hilbert spaces \cite{dBR66a,Sar94}.  We note that the augmented Kaczmarz algorithm shares some similarities to the relaxed Kaczmarz algorithm: both employ a relaxation parameter, and it is the limiting behavior of the relaxation parameter that yields the stability, in both cases, of the reconstructions in the presence of noise.  However, we will demonstrate in Proposition \ref{P:augmented-relaxed} that they are not equivalent methods.

We note here that there are strong connections between the Kaczmarz algorithm and frame theory \cite{KwMy01,Szw06,CzaTan13a,ChePow16a,HJW19a}.  Frames \cite{DS52a,DGM86a,Cas00a} have good stability properties with respect to additive noise:  if the noise is of finite energy (i.e. in $\ell^2$), then the frame reconstruction converges, and the error is proportional to the $\ell^2$-norm of the noise.  However, requiring the vectors $\{ \varphi_{n} \}$ to form a frame is a strong condition, and in general, stationary sequences do not form frames \cite{HW17a}.  Moreover, the Kaczmarz reconstruction has broad interest in applications \cite{GBH-70,EHL-81,SV-09a,NZZ15a,NSW16a,HKW19a}, and yields constructive proofs of the existence of Fourier series expansions for singular measures (\cite{Pol93,HW17a}, see also Subsection \ref{ssec:HS-BF}).

	\subsection{Series Representation of the Kaczmarz Algorithm} \label{ssec:SR-KA}
	
The Kaczmarz reconstruction given in Equation (\ref{eq:update}) can be expressed as a series, provided the sequence $\{ \vec{x}_{n} \}$ of approximations converges.  Following \cite{KwMy01}, we define the \emph{auxiliary sequence} $\{g_n\}_{n=0}^\infty$ by setting
\begin{align}
g_0 &= \varphi_0 \notag \\
g_n &= \varphi_n-\sum_{i=0}^{n-1}\langle \varphi_n,\varphi_i\rangle g_i. \label{eq:g_n}
\end{align}
It was shown in \cite{KwMy01} that
\[\vec{x}_n=\sum_{i=0}^n\langle \vec{x},g_i\rangle\varphi_i\]
and further that $\{ \varphi_{n} \}$ is effective if and only if $\{ g_{n} \}$ is a Parseval frame, meaning

\begin{align*}
\| \vec{x} \|^2 &= \sum_{n=0}^{\infty} | \langle \vec{x}, g_{n} \rangle |^2; \\
\vec{x} &= \sum_{n=0}^{\infty} \langle \vec{x}, g_{n} \rangle g_{n}.
\end{align*}

We may also define (as in \cite{HalSzw05}) the doubly-indexed sequence $\{\alpha_{n,j} : n \in \mathbb{N}_{0}, \ j = 0, \dots, n \}$ satisfying the property
\[\sum_{j=i}^n \alpha_{n,j}\langle \varphi_j,\varphi_i \rangle=\delta_{in}\]
for all $n\in \mathbb{N}$ and $i\leq n$.

We have that
\begin{align}
\lim_{N \to \infty} \vec{x}_{N} &= \sum_{n=0}^{\infty} \left( \sum_{j=0}^{n} \alpha_{n,j} \left( \langle \vec{x}, \varphi_{j} \rangle + \epsilon_{j} \right) \right) \varphi_{n} \label{eq:noisy-approximation} \\
&= \sum_{n=0}^{\infty} \left( \sum_{j=0}^{n} \alpha_{n,j}  \langle \vec{x}, \varphi_{j} \rangle \right) \varphi_{n} + \sum_{n=0}^{\infty} \left( \sum_{j=0}^{n} \alpha_{n,j}  \epsilon_{j} \right) \varphi_{n}  \label{eq:epsilon-sequence} \\
&=:  \vec{x}_{\infty} + E_{\epsilon}
\end{align}
and assuming both series on the right-hand side of Equation (\ref{eq:epsilon-sequence}) converge.  The identity in Equation (\ref{eq:noisy-approximation}) is a consequence of Lemma \ref{L:arbitrary} below.  If the sequence $\{ \varphi_{n} \}$ is effective, then the first series converges, and $\vec{x}_{\infty} = \vec{x}$.  Consequently:
\begin{Proposition}
Suppose the sequence $\{ \varphi_{n} \}$ is an effective sequence.  The sequence of approximations $\{ \vec{x}_{N} \}$ (with noise) converges if and only if the series $E_{\epsilon}$ converges.  Moreover, the reconstruction error is given by
\[ \lim_{N \to \infty} \| \vec{x}_{N} - \vec{x} \| = \| E_{\epsilon} \|. \]
\end{Proposition}

We will show in Subsection \ref{ssec:R-SS} that under rather restrictive assumptions on the noise $\{\epsilon_{n} \}$, we can obtain convergence of the series in Equation \eqref{eq:epsilon-sequence}, and can bound the norm on $E_{\epsilon}$.  However, in general, a fairly weak assumption on $\{\epsilon_{n}\}$ is not sufficient to guarantee convergence (Proposition \ref{P:divergence}).  Then, in Subsection \ref{ssec:R-AS}, we will show that be utilizing Abel summation rather than standard summation, we can obtain stability in the Kaczmarz reconstruction under much weaker assumptions.  Indeed, we state a general theorem that illustrates our main results; see Theorems \ref{Th:epsilon-boundary} and \ref{Th:epsilon-bounded} for the precise statements and proofs.

\begin{Theorem} \label{Th:overview}
Suppose $\{ \varphi_{n} \}_{n=0}^{\infty}$ is a stationary sequence which is effective, with singular spectral measure.  Under appropriate conditions on $\{ \epsilon_{n} \}$ (e.g. $\sum \epsilon_{n} z^{n} \in H^{\infty}(\mathbb{D})$), we have
\begin{equation} \label{eq:AS-noise}
\lim_{r \to 1^{-}} \sum_{n=0}^{\infty} r^{n} \left( \sum_{j=0}^{n} \alpha_{n,j} \left( \langle \vec{x}, \varphi_{j} \rangle + \epsilon_{j} \right) \right) \varphi_{n} = \vec{x}.
\end{equation}
In other words, by using Abel summation, the error term $E_{\epsilon} = 0$.
\end{Theorem}
We will also show in Subsection \ref{ssec:AS-AKA} that the Abel summation of the series in Equation (\ref{eq:AS-noise}) can be represented as an augmented Kaczmarz algorithm with relaxation parameter $r$.


	\subsection{Spectral Representation of the Kaczmarz Algorithm} \label{ssec:HS-BF}
	
	When the effective sequence $\{ \varphi_{n} \}$ is stationary, it possesses a spectral measure $\mu$ on $[0,1]$ as a result of Bochner's theorem.  The mapping $\varphi_{n} \mapsto e^{2 \pi i n x}$ then extends to an isometry from $H$ to $L^2(\mu)$; moreover, this isometry is a unitary when $\mu$ is singular.  The Kwapien-Mycielski theorem states that if the sequence $\{ \varphi_{n} \}$ is stationary and has dense span in $H$, then the following conditions are equivalent:  (i) the sequence is effective; (ii) the spectral measure $\mu$ is either Lebesgue measure or purely singular.
	
	Recent results elucidated the existence of a Fourier series for singular measures \cite{HW17a}, wherein the exponentials of positive-integer frequencies form a linearly dense collection of vectors for $L^2(\mathbb{T},\mu)$, where $\mu$ is a singular measure and $\mathbb{T}$ is the one-dimensional unit circle of the complex plane, parameterized by $e^{2\pi i x}$. (For the rest of this paper, we use the notation $L^2(\mu)=L^2(\mathbb{T},\mu)$, with the understanding that all functions of a single real variable are defined to be periodic.)
	
	It is well known that the Hardy space on the unit disc, $H^2(\mathbb{D})$ (notated as simply $H^2$), and $L^2_+(m)=\{f:f\in L^2(m) \text{ and }f(x)=\sum_{n=0}^\infty c_n e^{i2\pi nx}\text{ for some }c_n\in\ell^2\}$ (Lebesgue measure $m$) are isomorphic Hilbert spaces. This isomorphism is given by a change of variables $z\mapsto e^{2\pi i x}$ in the power series representation of $H^2$ functions. This isomorphism may also be thought of as the values the Hardy space function takes on the boundary. That is, for $f\in H^2$, and $z=re^{2\pi i x}$
	
	\begin{displaymath}
	\lim_{r\to 1^-}f(r e^{2\pi i x})=f^\star(x)\text{ for some }f^\star\in L^2(m).
	\end{displaymath}
	We call $f^\star$ the \textit{boundary function} of $f$. With respect to Lebesgue measure, for each $f\in H^2$, there is a guaranteed $f^\star\in L^2(m)$, as the two spaces are isomorphic. Subtleties occur when dealing with singular measures, however.	
	
	A similar isomorphism may be given between $L^2(\mu)$ and some subspace of $H^2$ \cite{HJW16a}. This subspace is most-easily described as the range of the \textit{normalized Cauchy transform}:
	
	\begin{displaymath}
	V_\mu f(z)=\frac{1}{\mu^+(z)}\int_0^1\frac{f(x)}{1-ze^{-2\pi i x}}d\mu(x) \text{ for }f\in L^2(\mu)
	\end{displaymath}
	where $\mu^+(z)$ is the unnormalized \textit{Cauchy transform} of the constant function $1$:
	
	\begin{displaymath}
	\mu^+(z)=\int_0^1\frac{1}{1-ze^{-2\pi i x}}d\mu(x).
	\end{displaymath}
	
Since we seek the Fourier series for this singular measure, we denote by 
	\begin{displaymath}
	\widehat{f}(j)=\int_0^1 f(x)e^{-2\pi i j x}d\mu(x)
	\end{displaymath}
the Fourier-Stieltjes transform of $f$. We call the sequence $\{\widehat{f}(j)\}$ the \textit{Fourier moments} of $f$. The Fourier moments of the constant function $\mathbf{1}(z)$ are notated as $\widehat{\mu}(j)$. These have the property that
	\begin{displaymath}
	\mu^+(z)=\sum_{n=0}^\infty \widehat{\mu}(n)z^n.
	\end{displaymath}

We also consider the \textit{Herglotz representation} of singular measures and their corresponding inner functions:
	
	\begin{displaymath}
	\text{Re}\left(\frac{1+b(z)}{1-b(z)}\right)=\int_0^1\frac{1-|z|^2}{|e^{2\pi i x}-z|^2}d\mu(x)
	\end{displaymath}
	where the function $b\in H^2$ is unique for each singular measure $\mu$. We refer to the function $b$ as \emph{corresponding to the singular measure} $\mu$; it has the properties that $b(0)=0$, $\lim_{r\to 1^-}|b(re^{2\pi i x})|=1$ $m-$almost everywhere, and $\lim_{r\to 1^-}b(re^{2\pi i x})=1$ $\mu-$almost everywhere \cite{Aleks89a,HerrThesis}. Additionally, we have
	\begin{displaymath}
	\mu^+(z)=\frac{1}{1-b(z)}\iff b(z)=1-\frac{1}{\mu^+(z)}.
	\end{displaymath}
	For $\varphi_n=e^{2\pi i n}$ we now recognize
	\begin{equation} \label{eq:alphas}
	1-b(z)=\sum_{n=0}^\infty \alpha_n z^n \iff b(z)=-\sum_{n=1}^\infty \alpha_n z^n.
	\end{equation}
	The following result from \cite{HW17a} regarding the sequence of Fourier moments summarizes much of the above:
\begin{thmalpha}\label{fourierserries} Assume $\mu$ is a singular Borel probability measure on $[0,1)$ and $\{\alpha_j\}$ is the corresponding sequence of coefficients given in Equation \eqref{eq:alphas}.  Then
\[f(x)=\sum_{n=0}^{\infty}\left(\sum_{j=0}^{n}\alpha_{n-j}\widehat{f}(j)\right)e^{2\pi inx},\]
where the convergence is in the $L^2(\mu)$-norm.
\end{thmalpha}

	\section{Main Results}
	
We assume throughout this section that $\{ \varphi_{n} \}_{n=0}^{\infty}$ is a linearly dense stationary sequence with singular spectral measure $\mu$.  Thus, the sequence is an effective sequence.  Moreover, we assume throughout that $b$ is the inner function corresponding to $\mu$ via the Herglotz representation, and thus the Normalized Cauchy Transform $V_{\mu} : L^2(\mu) \to H^2 \ominus bH^2$ is a unitary transformation.  We will also repeatedly refer to the sequence $\{ \alpha_{n} \}$ as defined in Equation \eqref{eq:alphas}.

Let us note that if the spectral measure $\mu$ is neither singular nor Lebesgue, then the sequence is not effective \cite{KwMy01} and so the question of stability is not well formulated in this case.  If the spectral measure is Lebesgue measure, then the sequence $\{ \varphi_{n} \}$ is an orthonormal basis, so we immediately obtain that the reconstruction error is $\| E_{\epsilon} \|^2 = \sum_{n=0}^{\infty} | \epsilon_{n}|^2$.
	
Since the sequence $\{ \varphi_{n} \} \subset H$ and $\{ e^{2 \pi i n x} \} \subset L^2(\mu)$ are unitarily equivalent, we will pass between them interchangeably. For the noise sequence $\{ \epsilon_{n} \}$ we define the formal series $\varepsilon(z) = \sum_{n=0}^{\infty} \epsilon_{n} z^{n}$; throughout this section we assume this series has radius of convergence at least $1$.

\subsection{The Kaczmarz Algorithm from Arbitrary Inputs} \label{ssec:AP}

The following provides the justification of the Kaczmarz reconstruction as a series representation.

\begin{Lemma} \label{L:arbitrary}
		For $\{c_n\}_{n=0}^\infty \subset \mathbb{C}$ we have by the Kaczmarz update: 
		\begin{displaymath}
		\vec{x}_N=\vec{x}_{N-1}+c_N\varphi_{N}-\langle \vec{x}_{N-1},\varphi_N\rangle \varphi_{N}
		\end{displaymath}
		but also: 
		\begin{displaymath}
		\vec{x}_N=\sum_{n=0}^N\bigg(\sum_{j=0}^n\alpha_{n-j}c_j\bigg)\varphi_n.
		\end{displaymath}
		Hence, the Kaczmarz reconstruction from an arbitrary sequence can be represented as a series.
	\end{Lemma}
	\begin{proof}
		The proof follows by induction. The initial case is trivial. Thus suppose
	\begin{displaymath}
	\begin{split}
	\vec{x}_N&=\vec{x}_{N-1}+c_N \varphi_{N}-\langle \vec{x}_{N-1},\varphi_{N}\rangle \varphi_{N}\\
	&=\sum_{n=0}^N \left(\sum_{j=0}^n\alpha_{n-j}c_j\right)\varphi_{n}.
	\end{split}
	\end{displaymath}
	By Kaczmarz we have
	\begin{displaymath}
	\begin{split}
	\vec{x}_{N+1}&=\vec{x}_N+c_{N+1} \varphi_{N+1}-\langle \vec{x}_N,\varphi_{N+1}\rangle \varphi_{N+1}\\
	&=\vec{x}_N+c_{N+1}\varphi_{N+1}-\bigg\langle\sum_{n=0}^N \left(\sum_{j=0}^n\alpha_{n-j}c_j\right)\varphi_n,\varphi_{N+1}\bigg\rangle \varphi_{N+1}\\
	&=\vec{x}_N+c_{N+1}\varphi_{N+1}-\sum_{n=0}^N \left(\sum_{j=0}^n\alpha_{n-j}c_j\right)\langle \varphi_n,\varphi_{N+1}\rangle \varphi_{N+1}.
	\end{split}
	\end{displaymath}
Hence,
	\begin{displaymath}
	\vec{x}_{N+1} =\sum_{n=0}^N \left(\sum_{j=0}^n\alpha_{n-j}c_j\right)\varphi_n+c_{N+1}\varphi_{N+1} 
	  -\sum_{n=0}^N \left(\sum_{j=0}^n\alpha_{n-j}c_j\right)\langle \varphi_n,\varphi_{N+1}\rangle \varphi_{N+1}
	\end{displaymath}
	We note the identity (given by the definition of $\{\alpha_k\}_{k=0}^\infty$ in Equation \eqref{eq:alphas})
	\begin{displaymath}
	-\sum_{n=0}^N\left(\sum_{j=0}^n \alpha_{n-j}c_j\right)\langle \varphi_n,\varphi_{N+1}\rangle=\sum_{j=0}^{N}\alpha_{N-j} c_j
	\end{displaymath}
	Thus we may write
	\begin{align*}
	\vec{x}_{N+1}&=\sum_{n=0}^N \left(\sum_{j=0}^n\alpha_{n-j}c_j\right)\varphi_n+c_{N+1}\varphi_{N+1} 
	 +\left(\sum_{j=0}^{N}\alpha_{N-j} c_j\right)\varphi_{N+1} \\
	 &=\sum_{n=0}^{N+1} \left(\sum_{j=0}^n\alpha_{n-j}c_j\right)\varphi_n
	\end{align*}
	\end{proof}

\subsection{Reconstruction using Standard Summation} \label{ssec:R-SS}

We begin by considering the reconstruction with noise using the auxiliary sequence $\{g_{n}\}$ as in Equation \eqref{eq:g_n}.  This does not fit within the iterative framework of the Kaczmarz algorithm, but we present it here to illustrate the distinctions between frame reconstructions and the Kaczmarz reconstruction.  In the noiseless case, we have by the Parseval frame condition on the auxiliary sequence that
\begin{equation} \label{Eq:noiseless-aux}
\vec{x} = \sum \langle \vec{x}, g_{n} \rangle g_{n} = \sum_{n=0}^{\infty} \left( \sum_{j=0}^{n} \alpha_{n-j} \langle \vec{x}, \varphi_{j} \rangle \right) g_{n}.
\end{equation}
Therefore, the error introduced by the noise is
\begin{equation} \label{Eq:g_n-error}
    E_{\epsilon,g} = \sum_{n=0}^{\infty} \left( \sum_{j=0}^{n} \alpha_{n-j}\epsilon_{j} \right) g_{n}
\end{equation}
provided the series converges.  If $\{ \epsilon_{n} \} \in \ell^2$, then $\{ \sum_{j=0}^{n} \alpha_{n-j} \epsilon_{j} \} \in \ell^2$ also, since
\begin{equation} \label{Eq:1-be}
    \sum_{n=0}^{\infty} \left( \sum_{j=0}^{n} \alpha_{n-j}\epsilon_{j} \right) z^n = (1-b(z)) \varepsilon(z) \in H^2(\mathbb{D}),
\end{equation}
and the norm satisfies the trivial estimate $\| (1 - b(z)) \varepsilon(z) \| \leq 2 \| \varepsilon(z) \|$.   Thus, we obtain:

\begin{Lemma}
If $\{\epsilon_n\}_{n=0}^\infty\in \ell^2$, then $\vec{y} = \sum_{n=0}^\infty \left( \sum_{j=0}^{n} \alpha_{n-j} ( \langle \vec{x} , \varphi_{j} \rangle + \epsilon_{j} ) \right) g_n$ converges, and 
\begin{equation*}
\| \vec{y} - \vec{x} \| \leq 2\|\epsilon_n\|_2
\end{equation*}
\end{Lemma}

The following is well-known concerning pseudocontinuable functions \cite{Clark72,Nik86a,Aleks89a,Pol93}, though we include the proof in the Appendix.

\begin{Proposition} \label{P:divergence}
For a singular measure $\mu$, there exists a sequence $\{\epsilon_{n} \} \in \ell^2$ such that the series $\sum \epsilon_{n} e^{2 \pi i n x}$ diverges in $L^2(\mu)$.  Moreover, there exists a sequence $\{ \epsilon_{n} \} \in \ell^2$ such that $\sum_{n=0}^{\infty} \left( \sum_{j=0}^{n} \alpha_{n-j} \epsilon_{j} \right) e^{2 \pi i n x}$ diverges in $L^2(\mu)$.
\end{Proposition}

In the context of the stability of frames, the assumption on the noise is $\{ \epsilon_{n} \} \in \ell^2$.  As a consequence of Proposition \ref{P:divergence}, however, we must make alternative assumptions on the noise to obtain convergence, as well as stability, of the Kaczmarz reconstruction in Equation \eqref{eq:epsilon-sequence}.  To help determine when the series $E_{\epsilon}$ converges, we utilize known facts concerning the convergence of Fourier series in $L^2(\mu)$. The foundational result in this regard is: if $f \in H^2 \ominus bH^2$, then its Fourier series converges in $L^2(\mu)$ \cite{Pol93,HW17a}.  The following is a known extension of this result, though we include the proof in the Appendix.

\begin{Lemma} \label{L:FS-convergence}
If $f(z) \in H^2 \ominus b^K H^2$ for some $K \in \mathbb{N}$, then the Fourier series for $f$ converges in $L^2(\mu)$, i.e. if $f(z) = \sum a_{n} z^{n}$, then $\sum a_{n} e^{2 \pi i n x}$ converges in $L^2(\mu)$.
\end{Lemma}

We can now use this to establish our main result concerning convergence and approximation of the error using the standard Kaczmarz reconstruction.

\begin{Theorem} \label{T:Wold} 
Suppose $\{ \varphi_{n} \} \subset H$ is a stationary sequence with singular spectral measure $\mu$.  Suppose $\vec{x} \in H$, and suppose the noise $\{ \epsilon_{n} \}$ is such that $(1-b)\varepsilon \in  H^2\ominus b^K H^2$ for some $K \in \mathbb{N}$.  Then, the reconstruction of $\vec{x}$ from the noisy data $\{ \langle \vec{x}, \varphi_{n} \rangle + \epsilon_{n}\} $ satisfies the following:
\begin{enumerate}
    \item the series $\vec{x}_{\infty}$ and $E_{\epsilon}$ as in Equation \eqref{eq:noisy-approximation} both converge;
    \item \begin{equation} \label{Eq:bk-estimate}
    \| E_{\epsilon} \|_{H} \leq 2 K^{1/2} \| \epsilon_{n} \|_{\ell^2}.
\end{equation}
\end{enumerate}
\end{Theorem}

\begin{proof} 
Since the spectral measure of $\{ \varphi_{n} \}$ is singular, it is an effective sequence, and thus $\vec{x}_{\infty} = \vec{x}$.

By identifying $\varphi_{n}$ with $z^{n}$ in Equation \eqref{eq:epsilon-sequence}, we identify  $E_{\epsilon}$ with $(1-b(z))\varepsilon(z)$ via Equation \eqref{Eq:1-be}.  Using the Wold decomposition of $H^2$ (\cite{HJW18a}), we write
\begin{equation} \label{Eq:Wold-decomp}
    (1 - b(z))\varepsilon(z) = \sum_{j=0}^{K-1} b^{j}(z)f_{j}(z)
\end{equation}
where $f_{j} \in H^2 \ominus bH^2$.  By Lemma \ref{L:FS-convergence}, we have that the Fourier series of $(1 - b(z))\varepsilon(z)$ converges in $L^{2}(\mu)$, so $E_{\epsilon}$ converges in $H$, and we obtain the norm estimate
\begin{align*}
\| E_{\epsilon} \|_{H} &= \| ((1 - b)\varepsilon)^{\star}\|_{\mu} \\
&\leq \sum_{j=0}^{K-1} \| f_{j}^{\star} \|_{L^2(\mu)} \\
&\leq K^{1/2} \left( \sum_{j=0}^{K-1} \| f_{j}^{\star} \|_{L^2(\mu)}^{2} \right)^{1/2}.
\end{align*}

We also have the norm equality
\begin{equation*}
\| (1 - b)\varepsilon \|_{H^2}^2 = \sum_{j=0}^{K-1} \| b^{j} f_{j} \|_{H^2}^2 = \sum_{j=0}^{K-1} \| f_{j}^{\star} \|_{L^2(\mu)}^{2}.
\end{equation*}
Combining these yields the inequality in \eqref{Eq:bk-estimate}.
\end{proof}

We can refine this result slightly by making extra assumptions on the $\{ \epsilon_{n} \}$.  For example, in the context of frames, we can quantify the error caused by the noise in terms of the frame bounds; but we can also quantify the error in terms of the cancellation in the frame expansion.  Indeed, for a frame $\{ h_{n} \} \in H$ with analysis operator $\Theta_{h}$, if $\{ \epsilon_{n} \}$ has the property that $\Theta_{h}^{*}(\epsilon_{n}) = 0$, then the reconstruction error is $0$.  This occurs exactly when $\{\epsilon_{n}\}$ is orthogonal to the range of $\Theta_{h}$.  Reconsidering the reconstruction using the auxiliary sequence $\{ g_{n} \}$, we obtain the following refinement.

\begin{Proposition}  \label{P:g_n-no-error}
If $\{ \epsilon_{n} \}$ is such that $\varepsilon(z) \in bH^2$, then $E_{\epsilon,g} = 0$.
\end{Proposition}

\begin{proof}
For the frame $\{ g_{n} \}$, the range of its analysis operator $\Theta_{g}(H) = H^2 \ominus bH^{2}$ (\cite{HW17a}).  Thus, the reconstruction using $g_{n}$ in Equation \eqref{Eq:g_n-error} is such that the coefficients satisfy $(1 - b(z)) \varepsilon (z) \in bH^2$, and thus sums to $0$.
\end{proof}

We cannot obtain such a clean statement for reconstruction from $\{ \varphi_{n} \}$, since the coefficients $\{ \langle \vec{x}, \varphi_{n} \rangle \}$ are not in a natural Hilbert space, and so there is not a natural analogue for the kernel of the ``synthesis'' operator.  However, for reconstruction using the standard Kaczmarz algorithm, we can use the boundary functions to obtain a similar result.  Indeed, if $E_{\epsilon}$ converges, and $\left[ (1 - b(z)) \varepsilon(z) \right]^{\star} = 0$, we immediately obtain $E_{\epsilon} = 0$.

\begin{Proposition}
Suppose $\{ \epsilon_{n} \}$ is such that $(1 - b)\varepsilon \in H^2 \ominus b^{K} H^2$ for some $K \in \mathbb{N}$. Then $E_{\epsilon} = 0$  if and only if and the Wold decomposition of $(1-b)\varepsilon$ as in Equation \eqref{Eq:Wold-decomp} satisfies $\sum_{j} f_{j} = 0$.
\end{Proposition}

\begin{proof}
This follows from the previous observation and the fact that the boundary
\begin{equation*}
    \left[ (1 - b) \varepsilon \right]^{\star} = \sum_{j} f_{j}^{\star} = 0
\end{equation*}
if and only if $\sum_{j} f_{j} = 0$ since $f_{j} \in H^2\ominus bH^2$.
\end{proof}



\subsection{Reconstruction using Abel Summation} \label{ssec:R-AS}

While the results of Theorem \ref{T:Wold} allowed us to bound the norm of the reconstruction error using standard summation methods, in analogy with the relaxed Kaczmarz algorithm, we now utilize the stabilizing power of Abel summation.

Therefore, instead of the summation in Equation \eqref{eq:noisy-approximation}, we consider the Abel sum:
\begin{equation} \label{eq:noisy-approximation-abel}
    \lim_{r \to 1^{-}} \sum_{n=0}^{\infty} r^{n} \left( \sum_{j=0}^{n} \alpha_{n-j} \left( \langle \vec{x}, \varphi_{j} \rangle + \epsilon_{j} \right) \right) \varphi_{n}
\end{equation}
provided the limit exists.  We will see that in several different scenarios on the noise $\{ \epsilon_{n} \}$, Abel summation in fact completely removes the noise from the reconstruction.

We first note that for $r \in (0,1)$, the convergence of the series in Equation \eqref{eq:noisy-approximation-abel} is guaranteed under minimal assumptions on $\{ \epsilon_{n} \}$ (though not the existence of the limit).

\begin{Lemma} \label{L:abel-sum}
Suppose $\{ \varphi_{n} \}_{n=0}^{\infty} \subset H$ is a stationary sequence with singular spectral measure $\mu$.  For $r \in (0,1)$, the series
\begin{equation}  \label{Eq:abel-sum}
\sum_{n=0}^{\infty} r^{n} \left( \sum_{j=0}^{n} \alpha_{n-j} \epsilon_{j} \right)  \varphi_{n}
\end{equation}
converges in $H$ provided $\limsup_{n \to \infty} \sqrt[n]{ | \epsilon_{n} | } \leq 1$.  Consequently, the series in Equation \eqref{eq:noisy-approximation-abel} converges.
\end{Lemma}

\begin{proof}
If we replace $\varphi_{n}$ with $e^{2 \pi i n x}$, the series in Equation \eqref{Eq:abel-sum} becomes:
\begin{equation} \label{Eq:sumonD}
\sum_{n=0}^\infty \bigg(\sum_{j=0}^n\alpha_{n-j}\epsilon_j\bigg)z^n=\left(1-b(z)\right)\epsilon(z)
\end{equation}
for $z\in\mathbb{D}$. Our assumption on $\{ \epsilon_{n} \}$ implies the right-hand side of Equation (\ref{Eq:sumonD}) is analytic on $\mathbb{D}$.  It now follows from uniform convergence that the series in Equation (\ref{Eq:abel-sum}) converges in $L^2(\mu)$.

Since $\{ \varphi_{n} \}$ is effective in $H$, we have
\begin{equation*}
    \lim_{r \to 1^{-}} \sum_{n=0}^{\infty} r^{n} \left( \sum_{j=0}^{n} \alpha_{n,j} \left( \langle \vec{x}, \varphi_{j} \rangle \right) \right) \varphi_{n} = \vec{x}.
\end{equation*}
\end{proof}

Recall that for a holomorphic function $f$ on the disc, we say that $f$ has a $L^2(\mu)$-boundary if there exists a $f^{\star} \in L^2(\mu)$ such that
\[ \lim_{r \to 1^{-}} f(r e^{2 \pi i x}) = f^{\star}(x) \]
with convergence in norm.  The following follows immediately from this definition.

\begin{Lemma} \label{L:obvious}
Suppose $\{ \varphi_{n} \}_{n=0}^{\infty} \subset H$ is a stationary sequence with singular spectral measure $\mu$.  The limit in Equation \eqref{eq:noisy-approximation-abel} exists if and only if $\left[(1-b(z))\varepsilon(z)\right]^{\star} \in L^2(\mu)$.  Moreover, the reconstruction error is
\begin{equation*}
    \lim_{r \to 1^{-}} \left\| \vec{x} - \sum_{n=0}^{\infty} r^{n} \left( \sum_{j=0}^{n} \alpha_{n-j} \left( \langle \vec{x}, \varphi_{j} \rangle + \epsilon_{j} \right) \right) \varphi_{n} \right\|_{H} = \left\| \left[(1-b(z))\varepsilon(z)\right]^{\star} \right\|_{L^2(\mu)}.
\end{equation*}
\end{Lemma}

We would like to have assumptions on $\{ \epsilon_{n} \}$ that guarantee convergence of the Abel sum.  This is easiest when the Abel sum of the noise term converges to $0$.  We present several different assumptions that yield this convergence  in Equation \eqref{eq:noisy-approximation-abel}; in these cases the reconstruction error is $0$.

\begin{Theorem} \label{Th:epsilon-boundary}
Suppose $\{ \varphi_{n} \} \subset H$ is a stationary sequence with singular spectral measure $\mu$.  Suppose $\vec{x} \in H$, and suppose the noise $\{ \epsilon_{n} \}$ is such that $\varepsilon^{\star} \in L^2(\mu)$ exists.

Then the following conditions hold:
\begin{enumerate}
\item 
\begin{equation} \label{Eq:Abel-noise}
\lim_{r\to 1^-} \left\|\sum_{n=0}^\infty r^n \left( \sum_{j=0}^n \alpha_{n-j}\epsilon_{j} \right) \varphi_{n} \right\|_{H} = 0;
\end{equation}
\item consequently, the limit in Equation \eqref{eq:noisy-approximation-abel} is $\vec{x}$.
\end{enumerate}
\end{Theorem}

\begin{proof}
As before, we express the sum in Equation \eqref{Eq:Abel-noise} in its spectral representation by replacing $\varphi_{n}$ with $e^{2 \pi i n x}$.  In so doing, we obtain
\begin{align*}
\lim_{r\to 1^-} \left\|\sum_{n=0}^\infty r^n \left( \sum_{j=0}^n \alpha_{n-j}\epsilon_{j} \right) \varphi_{n} \right\|^{2}_{H} &= \lim_{r \to 1^{-}} \int_{0}^{1} | 1 - b(r e^{2 \pi i x}) |^{2} | \varepsilon(r e^{2 \pi i x}) |^2 d \mu(x) \\
&= 0.
\end{align*}
The integral goes to 0 by Lebesgue's Theorem, since $|b(z)| < 1$, 
\begin{equation*}
\lim_{r \to 1^{-}} 1 - b(r e^{2 \pi i x}) = 0
\end{equation*}
in measure with respect to $\mu$, and $\varepsilon(r e^{2 \pi i x}) \to \varepsilon^{\star}$ in the $L^2(\mu)$-norm.
\end{proof}

\begin{Corollary} \label{C:epsilon}
Suppose $\varepsilon^\star\in L^p(\mu)$ for $2\leq p\leq \infty$. Then 
\begin{displaymath}
\lim_{r\to 1^-} \left\| \sum_{n=0}^\infty r^n \left( \sum_{j=0}^n \alpha_{n-j}\epsilon_{j} \right) \varphi_{n} \right\|_{H} =0.
\end{displaymath}
\end{Corollary}


The assumption that $\varepsilon^{\star}$ exists is a difficult assumption to check.  We turn now to assumptions that are easier to verify. To do so, we define a maximal function
\begin{equation} \label{Eq:maximal-function}
M_{\varepsilon}(x):=\sup_{0\leq r<1}|\varepsilon(re^{2\pi ix})|.
\end{equation}
Note that for $\varepsilon(z)\in H^\infty(\mathbb{D})$, $M_{\varepsilon}$ is essentially bounded with respect to $\mu$.

\begin{Lemma}  \label{L:epsilon-bounded}
Suppose $\{ \varphi_{n} \} \subset H$ is a stationary sequence with singular spectral measure $\mu$.  Suppose $\vec{x} \in H$, and suppose the noise $\{ \epsilon_{n} \}$ is such that $M_{\varepsilon} \in L^{2}(\mu)$.  Then the following conditions hold:
\begin{enumerate}
\item $$ \lim_{r\to 1^-} \left\| \sum_{n=0}^\infty r^n \left( \sum_{j=0}^n \alpha_{n-j}\epsilon_{j} \right) \varphi_{n} \right\|_{H}=0;$$
\item consequently, the limit in Equation \eqref{eq:noisy-approximation-abel} is $\vec{x}$.
\end{enumerate}
\end{Lemma}

\begin{proof}
For $r \in (0,1)$, we have
\begin{align*}
\left\| \sum_{n=0}^\infty r^n \left( \sum_{j=0}^n \alpha_{n-j}\epsilon_{j} \right) e^{ 2 \pi in x} \right\|^{2}_{L^2(\mu)} &= \int_{0}^{1} | 1 - b(r e^{2 \pi i x} ) |^2 | \varepsilon(r e^{2 \pi i x}) |^2 d \mu(x) \\
&\leq \int_{0}^{1} | 1 - b(r e^{2 \pi i x} ) |^2 | M_{\varepsilon}(x) |^2 d \mu(x).
\end{align*}
Thus, the norm goes to 0 as $r \to 1^{-}$ by Lebesgue's Theorem.
\end{proof}

\begin{Theorem} \label{Th:epsilon-bounded}
Suppose $\{ \varphi_{n} \} \subset H$ is a stationary sequence with singular spectral measure $\mu$.  Suppose $\vec{x} \in H$.  If $\varepsilon(z)$ satisfies any of the following conditions, then the conclusions of Lemma  \ref{L:epsilon-bounded} hold:
\begin{enumerate}
    \item $\varepsilon(z) \in H^{\infty}$;
    \item $M_{\varepsilon} \in L^{\infty}(\mu)$;
    \item for some $q>1$,
\[ \sup_{0<r<1} \int_{0}^{1} | \varepsilon(r e^{2 \pi i x}) |^{2q} d\mu(\xi) < \infty. \]
\end{enumerate}
\end{Theorem}

\begin{proof}
Items 1. and 2. follow immediately from Lemma \ref{L:epsilon-bounded}.  We prove Item 3. as follows.

As before, we have $1 - b(r e^{2 \pi i x}) \to 0$ boundedly in measure, so 
\begin{equation*}
\lim_{r \to 1^{-}} \int_{0}^{1} | 1 - b(r e^{2 \pi i x}) |^{2p} d \mu = 0
\end{equation*}
for any $1<p<\infty$.  Therefore by H\"{o}lder's Inequality,
\begin{align*}
\lim_{r \to 1^{-}} \bigg\|\sum_{n=0}^\infty r^{n} & \left( \sum_{j=0}^n \alpha_{n-j}\epsilon_{j} \right) e^{ 2 \pi i n x} \bigg\|_{L^2(\mu)}^2 =\lim_{r\to 1^-} \int_0^1 \left|\big(1-b(re^{2\pi ix})\big)\varepsilon(re^{2\pi ix}) \right|^2 d\mu(x)\\
&\leq \lim_{r\to 1^-} \left(\int_{0}^{1}  \left| \left(1-b(re^{2\pi ix}) \right) \right|^{2p} d \mu \right)^{1/p} \left( \int_{0}^{1} | \varepsilon(re^{2\pi ix})|^{2q} d\mu(x) \right)^{1/q} \\
&= 0.
\end{align*}
\end{proof}

The results stated in Theorems \ref{Th:epsilon-boundary} and \ref{Th:epsilon-bounded} are our most powerful applications of Abel summation, which state that not only is the series representation of the Kaczmarz algorithm as given in Equation \eqref{eq:noisy-approximation} Abel summable, but also Abel summation removes the corruption due to the noise $\{ \epsilon_{n} \}$.  In general, we would like weaker criteria on $\{ \epsilon_{n} \}$ for which the series is Abel summable, with the reconstruction error bounded.  This is more difficult for at least 2 reasons:  1) there are no known results that we are aware of that guarantee convergence of the Abel summation of Equation \eqref{eq:noisy-approximation}--except that the series converges or the tautology in Lemma \ref{L:obvious}; 2) as we've demonstrated earlier, there can not in general be a bound on the error term based on the $\ell^2$-norm of $\{ \epsilon_{n} \}$.  However, we can bound the error in two special cases.  The first is an immediate consequence of Theorem \ref{T:Wold}.

\begin{Proposition} \label{P:Wold-Abel} Suppose $(1-b)\varepsilon\in  H^2 \ominus b^{K} H^2$ for some $K \in \mathbb{N}$. Then for some constant $c_K$ dependent on $K$, we have that $\|\big((1-b)\varepsilon\big)^\star\|_{L^2(\mu)}\leq c_K\|(1-b)\varepsilon\|_{H^2}$. Therefore,
$$\lim_{r\to1^-}\Big\|\vec{x} - \sum_{n=0}^\infty r^n\Big(\sum_{j=0}^n\alpha_{n-j}\big( \langle \vec{x}, \varphi_{j} \rangle +\epsilon_{j} \big)\Big)\varphi_{n} \Big\|_{H} \leq c_K\big\|(1-b)\varepsilon\big\|_{H^2}.$$
\end{Proposition}

\begin{Proposition}
Suppose $\{ \epsilon_{n} \}$ is such that 
\begin{equation*}
    \sup_{0<r<1} \left( \int_{0}^{1} | 1 - b(r e^{2 \pi i x} ) |^2 | \varepsilon(r e^{2 \pi i x}) |^2 d \mu(x) \right)^{1/2} = B < \infty.
\end{equation*}
Then
\begin{equation} \label{Eq:weak-estimate}
    \limsup_{r \to 1^{-}} \left\| \vec{x} - \sum_{n=0}^\infty r^n\Big(\sum_{j=0}^n\alpha_{n-j}\big( \langle \vec{x}, \varphi_{j} \rangle +\epsilon_{j} \big)\Big)\varphi_{n} \right\|_{H} \leq B.
\end{equation}
In particular, if
\begin{equation*}
    \sup_{0<r<1} \left( \int_{0}^{1} | \varepsilon(r e^{2 \pi i x}) |^2 d \mu(x) \right)^{1/2} = C < \infty.
\end{equation*}
Then Inequality \eqref{Eq:weak-estimate} holds with $2C$ on the right hand side.
\end{Proposition}

Note that we are not concluding here that the series is Abel summable.

\begin{proof}
We have the estimate
\begin{align}
\limsup_{r \to 1^{-}} & \left\| \vec{x} - \sum_{n=0}^\infty r^n \Big(\sum_{j=0}^n\alpha_{n-j}\big( \langle \vec{x}, \varphi_{j} \rangle +\epsilon_{j} \big)\Big)\varphi_{n} \right\|_{H} \notag \\
& \hspace{1cm} \leq  \lim_{r \to 1^{-}} \left\| \vec{x} - \sum_{n=0}^\infty r^n\Big(\sum_{j=0}^n\alpha_{n-j}\big( \langle \vec{x}, \varphi_{j} \rangle \Big)\varphi_{n} \right\|_{H} \label{Eq:middle} \\ 
& \hspace{2cm} + \limsup_{r \to 1^{-}} \left( \int_{0}^{1} | 1 - b(r e^{2 \pi i x} ) |^2 | \varepsilon(r e^{2 \pi i x}) |^2 d \mu(x) \right)^{1/2} \label{Eq:last}
\end{align}
with the \eqref{Eq:middle} converging to $0$ and \eqref{Eq:last} bounded by $B$.
\end{proof}

\subsection{Abel Summation and the Augmented Kaczmarz Algorithm} \label{ssec:AS-AKA}

The series and spectral representations of the Kaczmarz algorithm have provided powerful methods for determining the convergence and reconstruction errors in the presence of noise.  Indeed, as we have seen, Abel summation eliminates the corruption of the noise in many cases.  However, we would like to express the Abel summation method in the form of a Kaczmarz update, which we do so here in the form of an augmented iteration.  An augmented Kaczmarz algorithm was introduced in \cite{Aboud2019}, for the purpose of dualizing the Kaczmarz reconstruction. We similarly define the \emph{augmented Kaczmarz algorithm} for expressing the Abel summation in Equation (\ref{Eq:abel-sum}), Lemma \ref{L:abel-sum}, in terms of the Kaczmarz update.  Given $\{c_n\}_{n=0}^\infty$, define the update as 
	\begin{align} 
	\vec{x}_n=\vec{x}_{n-1}+c_n\varphi_n-\langle \vec{x}_{n-1},\varphi_n\rangle\varphi_n \notag \\
	\vec{y}_n=\vec{y}_{n-1}+r^n\langle \vec{x}_n-\vec{x}_{n-1},\varphi_n\rangle\varphi_n. \label{P:augmented-abel}
	\end{align}
	
	\begin{Theorem} \label{Th:Aug-Kacz}
		For a fixed $r \in (0,1)$, the $N^{th}$ step in the augmented Kaczmarz algorithm corresponds exactly with the partial Abel sum:
		\begin{displaymath}
		\vec{y}_N=\sum_{n=0}^N r^n\bigg(\sum_{j=0}^n\alpha_{n-j}c_n\bigg)\varphi_n.
		\end{displaymath}
	\end{Theorem}
	
	\begin{proof}
		Again we proceed by induction with a trivial first case. Thus we assume
		\begin{displaymath}
		\vec{y}_{N-1}=\sum_{n=0}^{N-1}r^n\bigg(\sum_{j=0}^n\alpha_{n-j}c_j\bigg)\varphi_n.
		\end{displaymath}
		Next note that by Lemma \ref{L:arbitrary}
		\begin{displaymath}
		\vec{x}_N-\vec{x}_{N-1}=\sum_{j=0}^N\alpha_{n-j}c_j\varphi_N=\langle \vec{x}_N-\vec{x}_{N-1},\varphi_N\rangle\varphi_N
		\end{displaymath}
		since $\varphi_N$ is assumed to be normalized. Hence we have
		\begin{displaymath}
		\begin{split}
		\vec{y}_N&=\vec{y}_{N-1}+r^N\sum_{j=0}^N\alpha_{n-j}c_j\varphi_N=\sum_{n=0}^{N-1}r^n\bigg(\sum_{j=0}^n\alpha_{n-j}c_j\bigg)\varphi_n+r^N\sum_{j=0}^N\alpha_{n-j}c_j\varphi_N\\
		&=\sum_{n=0}^{N}r^n\bigg(\sum_{j=0}^n\alpha_{n-j}c_j\bigg)\varphi_n.
		\end{split}
		\end{displaymath}
	\end{proof}

\begin{Corollary}
Let $r \in (0,1)$ be fixed, and suppose the hypotheses of Lemma \ref{L:abel-sum} are met.  Then the sequence $\{\vec{y}_{n}\}$ as defined in Equation (\ref{P:augmented-abel}) converges for $c_{n} = \langle f, \varphi_{n} \rangle + \epsilon_{n}$.
\end{Corollary}
\begin{proof}
This follows from Lemma \ref{L:abel-sum} and Theorem \ref{Th:Aug-Kacz}.
\end{proof}

\subsection{The Relaxed Kaczmarz Algorithm and the Augmented Kaczmarz Algorithm} \label{ssec:RK-LSS}


We recall that the relaxed Kaczmarz algorithm utilizes a relaxation parameter $\omega$, typically with $\omega \in (0,2)$.  For a periodic sequence (of unit vectors) $\{\psi_{n}\}_{n=1}^{\infty}$ with period $N$, we can construct a sequence of approximations $\vec{y}_{m}$ from the data $\{ \langle \vec{x}, \psi_{n} \rangle \}_{n=1}^{N}$ as follows:
\begin{align}
    \vec{x}_{n} &= \vec{x}_{n-1} + \omega \langle \vec{x} - \vec{x}_{n-1} , \psi_{n} \rangle \psi_{n} \label{Eq:relaxed-update} \\
    \vec{y}_{m} &= \vec{x}_{mN} \notag
\end{align}
As in Lemma \ref{L:arbitrary}, we can replace $\langle \vec{x}, \psi_{k} \rangle$ with a (periodic) sequence $\{ c_{n} \}$ in the update \eqref{Eq:relaxed-update}.  It can be shown that $\{\vec{y}_{m}\}$ is a convergent sequence \cite{T-71,N-86}; this limit depends on $\omega$.  If we denote
\begin{equation*}
    \lim_{m \to \infty} \vec{y}_{m} = \vec{y}(\omega),
\end{equation*}
then we consider
\begin{equation*}
    \lim_{\omega \to 0} \vec{y}(\omega) = \vec{z}_{0}.
\end{equation*}
It can also be shown that this limit exists \cite{T-71,N-86}, and moreover, is the least-squares solution, i.e. $\vec{z}_{0}$ minimizes the quantity
\begin{equation*}
    \sum_{n=1}^{N} | \langle \vec{z}, \psi_{n} \rangle - c_{n}|^{2}.
\end{equation*}
We can think of the parameter $r$ in Abel summation as a relaxation parameter.  Then the limiting processes $\omega \to 0$ and $r \to 1$ are analogous; however, these are not equivalent.

\begin{Proposition} \label{P:augmented-relaxed}
The \emph{augmented Kaczmarz algorithm} is equal to the \emph{relaxed Kaczmarz algorithm} if and only if $\{\omega_n\}=\{r^n\}$ for fixed $r\in(0,1)$, and $\{\varphi_n\}$ is an orthogonal set.
\end{Proposition}

\begin{proof}
Observe that in the augmented Kaczmarz algorithm we may represent
\begin{equation} \label{eq:aug-1}
\vec{y}_N=\sum_{n=0}^N r^n \left(\sum_{j=0}^{n}\alpha_{n-j}c_j\right)\varphi_n
\end{equation}
where we note that each $\varphi_N$ has the augmented coefficient $r^N$ and no other powers of $r$. That is, $\vec{y}_N=k_N(r^N\varphi_N)+\text{lower order terms}$.

We also see that in the relaxed Kaczmarz algorithm
\begin{equation} \label{eq:aug-2}
\begin{split}
\vec{x}_N=\vec{x}_{N-1}+\omega_N\left(c_N-\langle \vec{x}_{N-1},\varphi_N\rangle \right)\varphi_N
\end{split}
\end{equation}
such that each $\varphi_N$ has the relaxed coefficient $\omega_N$. That is, $\vec{x}_N=\kappa_N(\omega^N\varphi_N)+\text{lower order terms}$. As such, if equations \eqref{eq:aug-1} and \eqref{eq:aug-2} were equivalent, we would have that $\omega_j=r^j$ for all $j$.

Rewriting equation \eqref{eq:aug-2} recursively dependent on the $(N-2)^{th}$ term, however, we receive
\begin{equation} \label{eq:aug-3}
\begin{split}
\vec{x}_N&=\vec{x}_{N-2}+\omega_{N-1}\left(c_{N-1}-\langle  \vec{x}_{N-2},\varphi_{N-1}\rangle \right)\varphi_{N-1}\\
& \hspace{1cm} +\omega_N\left(c_N-\big\langle \vec{x}_{N-2}+\omega_{N-1}\left(c_{N-1}-\langle \vec{x}_{N-2},\varphi_{N-1}\rangle \right)\varphi_{N-1},\varphi_N\big\rangle \right)\varphi_N\\
&=\vec{x}_{N-2}+\omega_{N-1}\left(c_{N-1}-\langle \vec{x}_{N-2},\varphi_{N-1}\rangle \right)\varphi_{N-1}\\
& \hspace{1cm} +\omega_N\left(c_N-\langle \vec{x}_{N-2},\varphi_N\rangle\right)\varphi_N-\omega_N\omega_{N-1}\big\langle\left(c_{N-1}-\langle \vec{x}_{N-2},\varphi_{N-1}\rangle \right)\varphi_{N-1},\varphi_N\big\rangle \varphi_N\\
&=K_1\omega_N\varphi_N-K_2\omega_N\omega_{N-1}\varphi_N=(K_1 r^N-K_2r^{2N-1})\varphi_N
\end{split}
\end{equation}
for $K_1,K_2$ constants dependent on $\{c_n\},\{\varphi_n\}$. Clearly, for equation \eqref{eq:aug-3} to equal equation \eqref{eq:aug-2}, then $K_2=0$, which would imply that
\begin{displaymath}
\big\langle\left(c_{N-1}-\langle \vec{x}_{N-2},\varphi_{N-1}\rangle \right)\varphi_{N-1},\varphi_N\big\rangle=0\iff \langle \varphi_{N-1},\varphi_N\rangle =0
\end{displaymath}
since $\{c_n\}$ is arbitrary.

We could repeat this procedure to write $\vec{x}_N$ dependent on $\vec{x}_{N-3}$ in which case we would have that $\langle \varphi_{N-2},\varphi_N\rangle=0$, and so on. We thus recover that $\langle \varphi_{N-j},\varphi_N\rangle =0$ for all $j=1,...,N$.

Hence for the choice $\omega_j=r^j$, the \emph{augmented Kaczmarz algorithm} is equal to the \emph{relaxed Kaczmarz algorithm} if and only if $\{\varphi_n\}$ is an orthogonal set.
\end{proof}
	
\begin{Remark}
Recall that in the typical relaxed Kaczmarz algorithm, $\omega$ is fixed. Thus it would be impossible for $\{\omega_n\}=\{r^n\}$ for $r\in (0,1)$, implying that the relaxed Kaczmarz algorithm with any fixed parameter $\omega$ is \textit{never} equal to the augmented Kaczmarz algorithm.  In practice, $\omega$ is allowed to vary with the iteration $n$, though establishing convergence results in this case has been difficult to do.
\end{Remark}

\section{Stability via Truncated Sequences}  \label{sec:S-FS}

We consider here the situation that we make weak or no assumptions on the noise $\{\epsilon_{n}\}$.  Our results here also apply when the sequence $\{ \varphi_{n} \}$ is not stationary (though we still assume effective).  Our basic idea is to truncate the sequence to $\{\varphi_{n}\}_{n=0}^{N}$, then periodize and reconstruct using the relaxed Kaczmarz algorithm as in Equation (\ref{Eq:relaxed-update}). The relaxed Kaczmarz algorithm provides stability in the presence of noise when the sequence $\{ \varphi_{n} \}$ is periodic.

For $N \in \mathbb{N}$, define the periodic sequence $\{ \varphi_{n}^{(N)} \}$ given by:
\begin{equation} \label{Eq:periodized-phi}
\varphi_{n}^{(N)} = \varphi_{n} \text{ for $n=1, \dots, N$ and } \varphi_{n}^{(N)} = \varphi_{k}^{(N)} \text{ when $n \equiv k \mod N$. } 
\end{equation}

Given the data $\{ \langle \vec{x}, \varphi_{n} \rangle \}_{n=1}^{N}$, we apply the standard periodic Kaczmarz algorithm as in Equation (\ref{eq:update}) to obtain
\begin{equation*}
    \lim_{n \to \infty} \vec{x}_{n} = \mathbb{P}_{N} \vec{x},
\end{equation*}
where $\mathbb{P}_{N}$ is the projection onto the span of $\{ \varphi_{1}, \dots, \varphi_{N}\}$.  This is a consequence of the fact that periodic sequences are effective for their span.  The simple but crucial observation here is the following: $\lim_{N \to \infty} \mathbb{P}_{N} \vec{x} = \vec{x}$ since $\{ \varphi_{n} \}_{n=0}^{\infty}$ has dense span in $H$.  What we will see here is a trade-off between noise removal and accuracy of the reconstruction were there no noise.

\subsection{Truncation of Effective Sequences in the Presence of Noise} \label{ssec:TSS}

We consider the situtation of Problem \ref{P:problem-statement}, though we do not require the sequence $\{ \varphi_{n} \}$ to be stationary. From now on, we let $\{\tilde{x}_n\}_{n\in\mathbb{N}}$ be the sequence of terms in the reconstruction of $x$ using the corrupted terms $\{\langle x,\varphi_n^{(N)}\rangle +\epsilon_n\}_{n\in\mathbb{N}}$.  Explicitly, this means
	\[\tilde{x}_1=\langle \vec{x},\varphi_1^{(N)}\rangle+\epsilon_1\varphi_1^{(N)},\hspace{1cm}
	\tilde{x}_{n+1}=\tilde{x}_n+(\langle \vec{x},\varphi_{n+1}^{(N)}\rangle +\epsilon_{n+1}-\langle \tilde{x}_n,\varphi_{n+1}^{(N)}\rangle)\varphi_{n+1}^{(N)}.\]
	We offer a simple estimate for the magnitude of the error in this reconstruction.

\begin{Proposition}  
Let $P_n$ denote the orthogonal projection onto the span of $\varphi_n$ as above, and let $T=\prod_{j=0}^{N-1} (I-P_{N-j})$.  Then for any $\vec{x}\in\mathcal{H}$,
\begin{displaymath}
\limsup_{m\to\infty}\|\mathbb{P}_N\vec{x}-\tilde{x}_{mN}\|\le \frac{1}{1-\|T\|_N}\left(|\epsilon_N|+\sum_{j=1}^{N-1}\|(I-P_N)...(I-P_{j+1}) \varphi_{j}^{(N)} \| |\epsilon_j|\right)
\end{displaymath}
where $\|\cdot\|_N$ denotes the operator norm on $\mathcal{B}(\mathcal{H}_N)$.
\end{Proposition}

	\begin{proof}
	Assume for simplicity, $\vec{x}\in\mathcal{H}_N$.  
	By definition of $\tilde{x}_n$,
	\begin{displaymath}
	\begin{split}
	\vec{x}-\tilde{x}_n &=\vec{x}-(\tilde{x}_{n-1}+\langle \vec{x}-\tilde{x}_{n-1},\varphi^{(N)}_{n}\rangle \varphi_n^{(N)}+\epsilon_n\varphi_{n}^{(N)})\\
	&=(I-P_n)(\vec{x}-\tilde{x}_{n-1})-\epsilon_n\varphi_n^{(N)}
	\end{split}
	\end{displaymath}
	If we start from $n=N$ and repeat this $N$ times, we obtain
	\begin{displaymath}
	\begin{split}
	    \vec{x}-\tilde{x}_{N}&=(I-P_{N})...(I-P_1)\vec{x}-\left(\epsilon_N\varphi_N^{(N)}+\sum_{j=1}^{N-1}(I-P_{N})...(I-P_{j+1})\epsilon_{j}\varphi_{j}^{(N)}\right)\\
	    &=T\vec{x}-\left(\epsilon_N\varphi_N^{(N)}+\sum_{j=1}^{N-1}(I-P_{N})...(I-P_{j+1})\epsilon_{j}\varphi_{j}^{(N)}\right)
	    \end{split}
	\end{displaymath}
    By the same reasoning, for any $m\ge2$,
    \begin{displaymath}
    \vec{x}-\tilde{x}_{mN}=T(\vec{x}-\tilde{x}_{(m-1)N})-\left(\epsilon_N\varphi_N^{(N)}+\sum_{j=1}^{N-1}(I-P_{N})...(I-P_{j+1})\epsilon_{j}\varphi_{j}^{(N)}\right).
    \end{displaymath}
    It then follows by induction that for all $m\in\mathbb{N}$
    \begin{displaymath}
        \vec{x}-\tilde{x}_{mN}=T^m\vec{x}-\sum_{k=0}^{m-1}T^k\left(\epsilon_N\varphi_N^{(N)}+\sum_{j=1}^{N-1}(I-P_{N})...(I-P_{j+1})\epsilon_{j}\varphi_{j}^{(N)}\right)
    \end{displaymath}

    Taking norms, we have
    \begin{displaymath}
    \begin{split}
        \|\vec{x}-\tilde{x}_{mN}\|&\le\|T\|_N^m\|\vec{x}\|+\left\|\sum_{k=0}^{m-1}T^k\left(\epsilon_N\varphi_N^{(N)}+\sum_{j=1}^{N-1}(I-P_N)...(I-P_{j+1})\epsilon_j\varphi_j^{(N)}\right)\right\|\\
        &\le\|T\|_N^m\|\vec{x}\|+\left(|\epsilon_N|+\sum_{j=1}^{N-1}\|(I-P_N)...(I-P_{j+1})\varphi_{j}^{(N)} \|  |\epsilon_j| \right) \left(\sum_{k=0}^\infty\|T\|^k_N\right)\\
        &=\|T\|_N^m\|\vec{x}\|+\left(|\epsilon_N|+\sum_{j=1}^{N-1}\|(I-P_N)...(I-P_{j+1}) \varphi_{j}^{(N)} \| |\epsilon_j| \right) \left(\frac{1}{1-\|T\|_N}\right)\\
    \end{split}
    \end{displaymath}
    Just as in the proof of the classical Kaczmarz algorithm (see e.g. \cite{HKW19a}), since $\{\varphi_n^{(N)}\}_{n=1}^{N}$ spans $\mathcal{H}_N$, we have $\|T\|_N<1$.  Thus if we take the limit as $m\to\infty$, we obtain the desired estimate.
	\end{proof}

Notice that if we bound the norm of the product of  the orthogonal projections above by 1, then the estimate above implies that
\[\limsup_{m\to\infty}\|\mathbb{P}_N\vec{x}-\tilde{x}_{mN}\|\le\frac{1}{1-\|T\|_N}\sum_{j=1}^N|\epsilon_j|.\]
We can improve this estimate if we apply the relaxed Kaczmarz algorithm to the truncated sequence.

\begin{Proposition}
For $N \in \mathbb{N}$, let $\{ \varphi_{n}^{(N)} \}$ denote the periodic sequence with period $N$ as in Equation (\ref{Eq:periodized-phi}), and $\{ \epsilon_{n}^{(N)} \}$ be periodized analogously.  For $\omega \in (0,2)$, define the approximations
\begin{align*}
    \vec{x}_{n}^{(N)}(\omega) &= \vec{x}_{n-1}^{(N)} + \omega \left( \langle \vec{x}, \varphi_{n}^{(N)} \rangle + \epsilon_{n}^{(N)} \right) \varphi_{n}^{(N)}, \\
    \vec{y}^{(N)}(\omega) &= \lim_{m \to \infty} \vec{x}^{(N)}_{Nm}(\omega), \\
    \vec{z}^{(N)} &= \lim_{\omega \to 0} \vec{y}^{(N)}(\omega).
\end{align*}
Then, $\| \mathbb{P}_{N} \vec{x} - \vec{z}^{(N)} \|^{2} \leq \dfrac{1}{A_{N}}  \sum_{n=1}^{N} | \epsilon_{n} |^2$ where $A_{N}$ is the lower frame bound of $\{ \varphi_{n} \}_{n=1}^{N}$.
\end{Proposition}

\begin{proof}
We have that $\vec{z}^{(N)}$ is the least-square solution to the system of equations
\begin{equation*}
    \langle \vec{z}, \varphi_{n} \rangle = \langle \vec{x}, \varphi_{n} \rangle + \epsilon_{n}, n=1, \dots, N.
\end{equation*}
We know that the solution $\vec{z}^{(N)}$ is given by the Moore-Penrose inverse $\Theta_{N}^{\dagger}$ of the data:
\begin{equation*}
    \vec{z}^{(N)} = \Theta_{N}^{\dagger} \begin{pmatrix} \langle \vec{x}, \varphi_{1} \rangle + \epsilon_{1} \\ \vdots \\ \langle \vec{x}, \varphi_{N} \rangle + \epsilon_{N} \end{pmatrix},
\end{equation*}
where $\Theta_{N}$ is the analysis operator of $\{\varphi_{n} \}_{n=1}^{N}$:
\begin{equation*}
    \Theta_{N} : H_{N} \to \mathbb{C}^{N}: \vec{x} \mapsto \begin{pmatrix} \langle \vec{x}, \varphi_{1} \rangle \\ \vdots \\ \langle \vec{x}, \varphi_{N} \rangle \end{pmatrix}.
\end{equation*}
If $A_{N}$ is the lower frame bound of $\{\varphi_{n}\}_{n=1}^{N}$, then $\Theta_{N}$ is bounded below by $\sqrt{A_{N}}$.  It follows that $\Theta_{N}^{\dagger}$ is bounded above by $\frac{1}{\sqrt{A_{N}}}$.  Since
\begin{equation*}
    \Theta_{N}^{\dagger} \begin{pmatrix} \langle \vec{x}, \varphi_{1} \rangle \\ \vdots \\ \langle \vec{x}, \varphi_{N} \rangle \end{pmatrix} = \mathbb{P}_{N} \vec{x},
\end{equation*}
the estimate now follows.
\end{proof}

\section{Appendix}

\begin{proof}[Proof of Proposition \ref{P:divergence}]
The first statement follows from the second, but an alternative proof can be found in \cite[Theorem 3.10]{DHW14}.  The second statement follows from Fatou's construction \cite[Section II.A]{Koo98}.
\end{proof}

\begin{proof}[Proof of Lemma \ref{L:FS-convergence}]
For the inner function $b^{K}$, let $\mu^{K}$ denote the corresponding singular measure.  By the Herglotz Representation, $\mu << \mu^{K}$.  Since by \cite{Pol93} every $f \in H^2 \ominus b^{K} H^2$ has convergent Fourier series in $L^2(\mu^{K})$, it follows that the Fourier series also converges in $L^2(\mu)$.
\end{proof}

\begin{acknowledgement}
Lee Przybylski and Eric Weber were supported in part by the National Science Foundation and National Geospatial-Intelligence Agency under award \#1830254.
\end{acknowledgement}

\newcommand{\etalchar}[1]{$^{#1}$}
\providecommand{\bysame}{\leavevmode\hbox to3em{\hrulefill}\thinspace}
\providecommand{\MR}{\relax\ifhmode\unskip\space\fi MR }
\providecommand{\MRhref}[2]{%
  \href{http://www.ams.org/mathscinet-getitem?mr=#1}{#2}
}
\providecommand{\href}[2]{#2}

\end{document}